\documentclass{article}

\usepackage[applemac]{inputenc}		
\usepackage{authblk}

\usepackage{amsmath}
\usepackage{amsthm}
\usepackage{amsxtra}
\usepackage{pxfonts}
\usepackage{MnSymbol}
\usepackage{marvosym}

\usepackage{braket}
\usepackage{tikz} 
\usetikzlibrary{matrix,arrows} 
\usepackage{url}
\usepackage{hyperref}

\newcommand{\lo}[1]{\raisebox{-0.1ex}{#1}\,}
\newcommand{\loo}[1]{\raisebox{-0.2ex}{#1}\,}
\newcommand{\Lo}[1]{\raisebox{-0.3ex}{#1}\,}

\newcommand{\LO}[1]{\raisebox{-0.5ex}{#1}\,}
\newcommand{\LOO}[1]{\raisebox{-0.6ex}{#1}\,}

\newcommand{\R}{\mathbb R}
\newcommand{\C}{\mathbb C}
\newcommand{\N}{\mathbb N}

\newcommand{\norm}[1]{\left\| #1 \right\|}

\newcommand{\D}[1]{\text{d}#1}
\newcommand{\I}{\text{i}}
\newcommand{\e}{\text{e}}

\newcommand{\mc}[1]{\mathcal{#1}}
\newcommand{\mf}[1]{\mathfrak{#1}}

\DeclareMathOperator{\dom}{dom}

\DeclareMathOperator{\cl}{cl}
\DeclareMathOperator{\der}{der}

\theoremstyle{definition}
\newtheorem{defn}{Definition}[section]

\theoremstyle{plain}
\newtheorem{prop}[defn]{Proposition}

\theoremstyle{remark}
\newtheorem{ej}[defn]{Example}
\newtheorem{rk}[defn]{Remark}

\title{On the geometry underlying a real \\ Lie algebra representation}
\author{Rodrigo Vargas Le-Bert\footnote{
\Letter {\tt rvargas@inst-mat.utalca.cl}. Supported by Fondecyt Postdoctoral Grant N\textordmasculine 3110045. 
} }
\affil{
Instituto de Matemática y Física, Universidad de Talca \\
Casilla 747, Talca, Chile
}
\date{June 1, 2012}

\begin{document}

\maketitle

\begin{abstract}
Let $G$ be a real Lie group with Lie algebra $\mf g$.
Given a unitary representation $\pi$ of $G$, one obtains by differentiation a representation $\D\pi$ of $\mf g$ by unbounded, skew-adjoint operators on $\mc H$. Representations of $\mf g$ admitting such a description are called \emph{integrable,} and they can be geometrically seen as the action of $\mf g$ by derivations on the algebra of representative functions $g\mapsto\langle\xi,\pi(g)\eta\rangle$, which are  naturally defined on the homogeneous space $M=G/\ker\pi$. In other words, integrable representations of a real Lie algebra can always be seen as realizations of that algebra by vector fields on a homogeneous manifold.
Here we show how to use the coproduct of the universal enveloping algebra $\mc E(\mf g)$ to generalize this to representations which are not necessarily integrable. The  geometry now playing the role of $M$ is  a locally homogeneous space. This provides the basis for a geometric approach to integrability questions regarding Lie algebra representations.
\end{abstract}

%

\setcounter{tocdepth}{2}
\tableofcontents

\section{Introduction}

%
%

Let $G$ be the simply connected, real Lie group with Lie algebra $\mf g$. Recall that the \emph{group algebra} $\C[G]$ is the involutive algebra consisting of finite, complex linear combinations of elements of $G$ with product and involution  obtained by linearly extending the maps $g\otimes h\mapsto gh$ and $g\mapsto g^{-1}$, respectively. 
We can also define a complex conjugation on $\C[G]$ by anti-linearly extending $\overline g= g$. 
The infinitesimal version of $\C[G]$ is the \emph{universal enveloping algebra} $\mc E = \mc E(\mf g)$, 
defined as the quotient of the complex tensor algebra $\mc T(\mf g) = \left(\bigoplus_{n\in\N} \mf g^{\otimes n}\right)\otimes\C$ by the ideal generated by the elements
\[
XY - YX - [X,Y],\quad X,Y\in\mf g.
\]
It comes equipped with the involution 
and complex conjugation 
obtained by linear extension from
\(
( X_1\cdots X_n)^\dagger =  (-1)^nX_n\cdots X_1\lo,
\)
and anti-linear extension from
\(
\overline{X_1\cdots X_n} = X_1\cdots X_n\lo,
\)
respectively.
The subalgebra of real elements with respect to this conjugation will be denoted by $\mc E_\R = \mc E_\R(\mf g)$.

\subsection{Integrable representations}


Consider a representation $\pi$ of $G$ by unitary operators on a Hilbert space $\mc H$. Given $g\in G$ and $\xi\in\mc H$, it is sometimes  convenient to write $g\xi$ instead of $\pi(g)\xi$, and consequently we can also think of $\mc H$ as a $\C[G]$-module. 
In what follows we will use interchangeably the languages of representation and module theory.

Let $M=G/H$, with $H=\ker\pi$.
We want to understand $\mc H$ as a submodule of $C(M)$, which is naturally a $\C[G]$-module via
\[
gf(x) = f(g^{-1}x),\quad f\in C(M).
\]
In order to do so,
suppose that there is a distinguished cyclic vector $\nu\in\mc H$, i.e.\ $\mc H = \overline{\C[G]\nu}$, and an anti-unitary operator $J:\mc H\rightarrow\mc H$ whose action commutes with that of $G$.  Then,
%
%
to each $\xi\in\mc H$,  associate the representative function
\[
e_\xi: gH\in M\mapsto \langle J\xi, g\nu\rangle \in\C.
\]
Since $ge_\xi = e_{g\xi}$\loo, we see that $\xi\mapsto e_\xi$ does the job of exposing $\mc H$ as a submodule of $C(M)$.
\begin{rk} \label{regarding J}
The anti-unitary $J$ is needed in order for $\xi\mapsto e_\xi$ to be linear. Its use can be easily avoided in several ways, such as for instance working with the \emph{anti-linear} $\C[G]$-module associated to $\pi$, or  considering the anti-holomorphic ``representative functions'' $g\mapsto\langle g\nu,\xi\rangle$. Instead of using any such non-standard convention, we have decided to restrict ourselves to modules admitting an anti-unitary $J$ as above.
\end{rk}

Now, let $\mc H^\infty\subseteq\mc H$ be the subspace of \emph{smooth vectors} for $G$, i.e.\ those $\xi\in\mc H$ such that the map
\(
g\in G\mapsto g\xi \in \mc H
\)
is smooth.\footnote{ 
$\mc H^\infty$ is also called the \emph{Gårding space} of $\pi$.
As Harish-Chandra pointed out~\cite{m:Hari53}, the space of \emph{analytic vectors} (see also
\cite{m:Nels59, m:Schm90}) is more adequate in studying representations of $G$. 
By contrast, when studying non-integrable representations of $\mf g$ the right domain is  $\mc H^\infty$ (representations having a dense subspace of analytic vectors are automatically integrable~\cite{m:Magy92}, even if $\mf g$ is infinite dimensional~\cite{m:Neeb11}).
}
Then, given $X\in\mf g$ and $\xi\in\mc H^\infty$, we can define
\[
\D\pi(X)\xi = X\xi = \lim_{t\rightarrow 0} \frac{1}{t}\bigl( \exp(tX)\xi - \xi \bigr),
\]
where $\exp:\mf g\rightarrow G$ is the exponential map.
This way, the unitary representation $\pi$ of $G$ gives rise to a representation $\D\pi$ of $\mf g$ by unbounded, skew-adjoint operators, thus naturally turning $\mc H^\infty$ into an $\mc E$-module. We will also say, under this circumstances, that $\mc H$  is an \emph{unbounded} $\mc E$-module.
As above, we can expose $\mc H^\infty$ as a submodule of $C^\infty(M)$, by identifying it with $\set{e_\xi | \xi\in\mc H^\infty}$. Now, the action of $\mf g$ is by derivation along the vector fields which generate the flows $gH\in M\mapsto \e^{tX}gH\in M$.

It will be helpful to be more precise regarding the notion of representation by unbounded operators, which we do in the next subsection. For the time being, however, we anticipate the following definition.
\begin{defn}
A representation of $\mf g$ is \emph{integrable} if it is of the form $X\mapsto \D\pi(X)|_{\mc F}$\loo, with $\pi$ a unitary representation of $G$ and $\mc F\subseteq \mc H^\infty$.
\end{defn}

\subsection{Unbounded operator algebras}

For a complete exposition of the material in this and the next subsection see~\cite{m:Schm90}. Here we follow~\cite{m:Varg11}.

Let  $\mc F$ be a complex vector space sitting densely inside a Hilbert space $\mc H$.

\begin{defn}
We write $\mc L^\dagger(\mc F)$ for the algebra
of linear operators $A:\mc F\rightarrow \mc F$ such that $\mc F\subseteq \dom(A^*)$.
It comes equipped with the involution $A^\dagger = A^*|_{\mc F}$\Lo. 
\end{defn}
\begin{defn}
We write $\mc C^*(\mc F)$ for the set of closed operators (in the sense of having closed graph) $A$ on $\mc H$ which satisfy $A|_{\mc F}\in\mc L^\dagger(\mc F)$.
\end{defn}

Analytically speaking, closed operators are better behaved than arbitrary ones, but this is not so from an algebraic viewpoint. Indeed,
given two closed, densely defined operators $A$ and $B$ on $\mc H$, we define---whenever it makes sense---their \emph{strong sum}  by
\[
A+B = \text{the closure of } A|_{\mc D} + B|_{\mc D}\loo,\quad \mc D = \dom A\cap \dom B.
\]
Analogously, we define their \emph{strong product} by
\[
AB = \text{the closure of } AB|_{\mc D}\loo,\quad  \mc D = B^{-1}\dom A.
\]
As it turns out, strong operations are always well-defined for elements of $\mc C^*(\mc F)$, but  the axioms of associativity and distributivity might well not hold.

\begin{defn}
An \emph{unbounded representation} of a complex involutive algebra $\mc A$ is an involutive algebra morphism $\pi:\mc A\rightarrow \mc L^\dagger(\mc F)$, where $\mc F$ sits densely in a Hilbert space $\mc H$. Under this circumstances, we also say that $\mc H$ is an unbounded $\mc A$-module. Finally, we say that $\pi$ is \emph{cyclic} if there exists a so-called \emph{cyclic vector} $\nu\in\mc F$ such that $\mc F = \mc A\nu$.
\end{defn}

Recall that there is a natural order defined on any involutive algebra $\mc A$: the one with positive cone generated by the elements of the form $A^\dagger A$. This, in turn, induces an order on the algebraic dual of $\mc A$, which will be written $\mc A^\dagger$: an $\omega\in\mc A^\dagger$ is positive if, and only if, it is positive on positive elements. We say that such an $\omega$ is a \emph{state} if, besides being positive, it is normalized by $\omega(1)=1$. We have introduced all this notions just to recall that there is a correspondence between cyclic representations  and states.
Given a representation $\pi:\mc A\rightarrow \mc L^\dagger(\mc F)$, one obtains a state $\omega\in\mc A^\dagger$ from any vector $\nu\in\mc F$ with $\norm\nu = 1$ by the formula
\[
\omega(A) = \langle \nu,A\nu\rangle,\quad A\in\mc A.
\]
The GNS construction allows for a recovery of both $\pi|_{\mc A\nu}$ (modulo unitary conjugation) and $\nu$ (modulo a phase factor) from $\omega$, thus establishing the afforementioned correspondence. We proceed to describe it briefly.

Let $\omega$ be a state of $\mc A$ and consider the set
\(
\mc I = \set{ A\in \mc A | \omega(A^\dagger A)=0 }.
\)
Using the Cauchy-Schwartz inequality it is easily seen that $\mc I$ is a left ideal. The quotient $\mc F = \mc A/\mc I$ is densely embedded in its Hilbert space completion $\mc H$ with respect to the scalar product
\[
\left\langle [A], [B]\right\rangle = \omega(A^\dagger B),\quad A,B\in\mc A,
\]
where $[\cdot]$ denotes the equivalence class in $\mc F$ of its argument.
The GNS representation $\pi$ is simply given by the canonical left $\mc A$-module structure of $\mc F$. It admits the cyclic vector $\nu = [1]\in\mc F$.

\subsection{Non-integrable representations}

As we saw above, integrable $\mc E$-modules are, essentially, spaces of smooth functions on homogeneous manifolds.
Our objective in this paper is to work out the analogous geometric point of view for non-integrable representations. 
Thus, it is pertinent to discuss  briefly the phenomenon of non-integrability.

%

From an operator algebraic point of view, integrability is just good behaviour, in the following sense. 
Consider a representation $\pi:\mc E\rightarrow \mc L^\dagger(\mc F)$.
Operators in $\mc L^\dagger(\mc F)$ are closable, because their adjoints are densely defined. Denote the closure of (the image under $\pi$ of) $E\in\mc E$ by $\cl(E)$. 
Then,  integrability of $\pi$ is equivalent to:
\begin{enumerate}
\item $\cl(E+F) = \cl(E) + \cl (F)$,
\item $\cl(EF) = \cl(E) \cl(F)$,
\item $\cl(E^\dagger) = \cl(E)^*$,
\end{enumerate}
where on the right hand sides strong operations  are meant (the last property actually implies the first two). Hence, integrability implies that $\cl(1+E^\dagger E)$ is invertible, for all $E\in\mc E$. Let us mention, by the way, that this relationship between integrability and invertibility of ``uniformly positive'' elements is deep: it can be shown~\cite{m:Varg11}, whenever
\[
\mc S = \text{the multiplicative set generated by } \Set{1+E^\dagger E | E\in\mc E}
\]
satisfies the Ore condition, that integrable representations are in bijection with representations of the Ore localization $\mc E\mc S^{-1}$. 

An important question is whether a given representation admits an integrable extension. 
One can extend the algebra, the module, or both.
Thus, the integrable extension problem in its full generality consists in completing the following diagram:
\begin{center} \begin{tikzpicture}[description/.style={fill=white,inner sep=2pt}] 
\matrix (m) [matrix of math nodes, row sep=2em, 
column sep=2.5em, text height=1.5ex, text depth=0.25ex] 
{ \tilde{\mc E} & \mc C^*(\tilde{\mc F}) \\ 
\mc E & \mc L^\dagger(\mc F), \\ }; 
\path[->,font=\scriptsize] 
(m-2-1) edge (m-2-2) 
(m-1-2) edge  (m-2-2)
(m-1-1) edge [dashed]  (m-1-2)
(m-1-1) edge [dashed]  (m-2-1);
\end{tikzpicture} \end{center} 
with $\tilde{\mc E}\supseteq \mc E$ and $\tilde{\mc F} \supseteq \mc F$. 
The image of $\tilde{\mc E}$ under the upper arrow must be a subset of $\mc C^*(\tilde{\mc F})$ which is an algebra with respect to the strong operations, all arrows must be involutive algebra morphisms, and the diagram must commute.
We do not necessarily require that $\tilde{\mc F}\subseteq \mc H$. Again under the hypothesis that the set $\mc S$ above satisfies the Ore condition, cyclic representations always admit an integrable extension with $\tilde{\mc E} = \mc E$, see~\cite{m:Varg11}.

\section{Algebraic construction}

\subsection{Algebraic structure of $\mc E^\dagger$}

In this subsection we set up the needed algebraic preliminaries regarding the dual $\mc E^\dagger$ of the enveloping algebra $\mc E=\mc E(\mf g)$. 
A full exposition is found in~\cite{m:Dixm77}.
We will also make use of some coalgebra notions, for which we mention~\cite{m:Mont93}.

The universal enveloping algebra $\mc E$ is not only an algebra: it comes naturally endowed with a coproduct, too. 
Recall that coproducts are dual to products, so that a coproduct is a linear function
\(
\Delta: \mc E \rightarrow \mc E\otimes \mc E
\)
which is coassociative, meaning that the following diagram commutes:
\begin{center} \begin{tikzpicture}[description/.style={fill=white,inner sep=2pt}] 
\matrix (m) [matrix of math nodes, row sep=2.5em, 
column sep=2.5em, text height=1.5ex, text depth=0.25ex] 
{ \mc E & \mc E\otimes\mc E \\ 
\mc E\otimes\mc E & \mc E\otimes\mc E\otimes\mc E. \\ }; 
\path[->,font=\scriptsize] 
(m-1-1) edge node[auto] {$\Delta$} (m-1-2) 
(m-1-1) edge node[auto] {$\Delta$} (m-2-1)
(m-1-2) edge node[auto] {$\text{Id}\otimes\Delta$} (m-2-2)
(m-2-1) edge node[auto] {$\Delta\otimes\text{Id}$} (m-2-2);
\end{tikzpicture} \end{center} 
In order to define the coproduct of $\mc E$, we define first a coproduct $\Delta$
on the tensor algebra $\mc T(\mf g)$, by extending the linear map
\[
X \in\mf g\mapsto 1\otimes X + X\otimes 1\in\mc T(\mf g)\otimes \mc T(\mf g)
\]
to all of $\mc T(\mf g)$ as an algebra morphism (we are appealing here to the universal property of tensor algebras).
Now, the ideal $\mc I\subseteq \mc T(\mf g)$ generated by the elements
\(
XY - YX - [X,Y], 
\)
with $X, Y\in\mf g$,
is also a \emph{coideal}, i.e.\ 
\[
\Delta\mc I\subseteq \mc I\otimes\mc T(\mf g) + \mc T(\mf g)\otimes\mc I.
\]
Thus, it can be readily seen that $\Delta$ passes to the quotient $\mc E = \mc T(\mf g)/\mc I$.

It will be helpful to have an explicit formula for $\Delta$. Let us choose a basis $(X_1\lo,\dots,X_n)$ of $\mf g$ and use the multi-index notation
\(
X^\alpha = X_1^{\alpha_1}\cdots X_n^{\alpha_n}\Lo. 
\)
Using the fact that $\Delta$ is an algebra morphism together with the binomial formula, we see that
\[
\Delta X^\alpha = \sum_{\beta\leq\alpha} {\alpha\choose\beta} X^\beta\otimes X^{\alpha-\beta},
\]
where ${\alpha\choose\beta} = \frac{\alpha!}{\beta!(\alpha-\beta)!} = {\alpha_1\choose\beta_1} \cdots {\alpha_d\choose\beta_d}$\LO.

For us, the importance of the coproduct  resides in the fact that it enables one to turn the algebraic dual $\mc E^\dagger$ into an algebra, with product given by
\[
ab = (a\otimes b)\Delta,\quad a,b:\mc E\rightarrow \C.
\]
Recall that $\Delta$ is cocommutative, meaning that the following diagram commutes:
\begin{center} \begin{tikzpicture}[description/.style={fill=white,inner sep=2pt}] 
\matrix (m) [matrix of math nodes, row sep=.5em, 
column sep=3em, text height=1.5ex, text depth=0.25ex] 
{ 		& \mc E\otimes\mc E		\\
\mc E	&  					\\ 
 		& \mc E\otimes\mc E,	\\}; 
\path[->,font=\scriptsize] 
(m-1-2) 	edge node[auto] {$\tau$} 	(m-3-2) 
(m-2-1) 	edge node[above] {$\Delta$} 	(m-1-2)
(m-2-1) 	edge node[below] {$\Delta$} 	(m-3-2);
\end{tikzpicture} \end{center} 
where $\tau$ is obtained by linear extension from $E\otimes F\mapsto F\otimes E$, with $E,F\in\mc E$.
This implies that $\mc E^\dagger$ is a commutative complex algebra.
It is naturally an $\mc E$-module, too, with
\begin{equation} \label{Da}
Ea(F) = a\bigl(\overline{E}{}^\dagger F\bigr),\quad E,F\in\mc E,\ a\in\mc E^\dagger.
\end{equation}
The main properties of $\mc E^\dagger$ are given in the next two propositions.

\begin{prop} \label{g = derivations of E'}
The action \eqref{Da} defines a morphism $\mf g\rightarrow \der(\mc E^\dagger)$, where
$\der(\mc E^\dagger)$ stands for the Lie algebra of linear maps $\delta:\mc E^\dagger\rightarrow \mc E^\dagger$ such that
\[
(\forall a,b\in\mc E^\dagger)\ \delta(ab)=\delta(a)b+a\delta(b).
\]
\end{prop}
\begin{proof}
Indeed, let $X\in\mf g$ and $a,b\in \mc E^\dagger$. One has that
\[ 
X(ab) = X((a\otimes b)\circ \Delta) = (a\otimes b) \circ \Delta \circ X^\dagger(\cdot).
\] 
On the other hand, given $E\in\mc E$,
\begin{align*}
\Delta(X^\dagger E) &= \Delta (X^\dagger) \Delta(E) = (X^\dagger \otimes 1 + 1\otimes X^\dagger)\Delta(E) \\
	&= \bigl( \bigl(X^\dagger(\cdot)\otimes \text{Id} + \text{Id}\otimes X^\dagger(\cdot) \bigr)\circ \Delta \bigr)(E).
\end{align*}
Thus, replacing above,
\begin{align*}
X(ab) &= (a\otimes b) \circ \bigl(X^\dagger(\cdot)\otimes \text{Id} + \text{Id}\otimes X^\dagger(\cdot) \bigr)\circ \Delta \\
	&= (Xa\otimes b + a\otimes Xb)\circ\Delta 
	= (Xa)b + a(Xb). \qedhere
\end{align*}
\end{proof}

\begin{prop} \label{E^dagger = power series}
Let $X_1\lo,\dots, X_n$ be a basis of $\mf g$. 
One has that 
\(
\mc E^\dagger \cong \C\lsem x_1,\dots,x_n \rsem, 
\)
with algebra isomorphism given by 
\[
a\in\mc E^\dagger\mapsto f_a\in\C\lsem x_1,\dots x_n\rsem,\quad f_a(x) = \sum_{\alpha\in\N^n} \frac{x^\alpha}{\alpha!} a(X^\alpha). \qedhere
\]
\end{prop}
\begin{proof}
Indeed, given $a,b\in \mc E^\dagger$, 
\begin{align*}
f_{ab}(x) &= \sum_{\alpha} \frac{x^\alpha}{\alpha!} ab(X^\alpha)
	= \sum_\alpha \frac{x^\alpha}{\alpha!} (a\otimes b)(\Delta X^\alpha)
	= \sum_{\alpha} \frac{x^\alpha}{\alpha!} (a\otimes b)\left( \sum_{\beta\leq\alpha} {\alpha\choose\beta} X^\beta\otimes X^{\alpha-\beta}\right) \\
	&= \sum_\beta\sum_\gamma \frac{x^\beta x^\gamma}{\beta!\gamma!} a(X^\beta)b(X^\gamma) = f_a(x)f_b(x). \qedhere
\end{align*}
\end{proof}
\begin{rk} \label{E' = formal taylor series}
Observe that 
\[
ab(1) = (a\otimes b)\Delta 1 = a(1)b(1).
\]
Thus, evaluation at $1\in\mc E$ is a character of $\mc E^\dagger$, which we will denote by $\chi_0$\loo,
and elements of $\mc E^\dagger$ are, essentially, formal Taylor expansions around $\chi_0$\loo. 
\end{rk}

The fact that $\mc E^\dagger$ is a formal power series algebra suggests endowing it with an involution, corresponding to complex conjugation of the coefficients. Explicitly, this is given by
\[
a^\dagger(E+\I F) = \overline{a(E)} + \I \overline{a(F)}, \quad E,F\in\mc E_\R\loo.
\]
With respect to this conjugation, $\chi_0$ becomes a \emph{real} character:
\[
\chi_0(a^\dagger) = a^\dagger (1) = \overline{a(1)} = \overline{\chi_0(a)}.
\]
\begin{prop} \label{(Da)^dagger = bar D a^dagger}
For all $a\in\mc E^\dagger$ and $E\in\mc E$, $(Ea)^\dagger = \overline Ea^\dagger$.
\end{prop}
\begin{proof}
For $E,F\in\mc E_\R$\raisebox{-0.3ex},\, one has
\(
(Ea)^\dagger(F) = \overline{ a(E^\dagger F) } = Ea^\dagger (F), 
\)
and the general case follows immediately.
\end{proof}

\subsection{$M$ as an affine scheme}

Let $\pi:\mc E\rightarrow \mc L^\dagger(\mc F)$ be a representation of $\mc E=\mc E(\mf g)$ admitting a cyclic vector $\nu\in\mc F$, where $\mf g$ is a real Lie algebra and $\mc F$ is a dense subspace of a Hilbert space~$\mc H$. We will suppose that there exists an anti-unitary  $J:\mc H\rightarrow \mc H$ commuting with $\pi$ and such that $J\nu=\nu$, even though this assumption is superfluous, see Remark~\ref{regarding J}.
Recall that we  look for some naturally defined space $M$ together with a realization of $\mf g$ by vector fields on it, such that the corresponding coordinate algebra~$\mc A$, which is an $\mc E$-module,  contains $\mc F$ as a submodule. In this subsection
we realize $\mc A$ as a subalgebra of $\mc E^\dagger$.

Consider the linear map
\[
\xi\in\mc F \mapsto e_\xi\in\mc E^\dagger,\quad e_\xi(E) = \langle J\xi, E\nu\rangle,
\]
and let $\mc A$ be the involutive subalgebra of $\mc E^\dagger$ generated by $\set{e_\xi|\xi\in\mc F}$. 
We define
\[
M = \Set{ \chi\in\mc A^\dagger | \chi(ab) = \chi(a)\chi(b),\ \chi(a^\dagger) = \overline{\chi(a)}}.
\]
Note that $\chi_0\in M$. 
\begin{rk}
The map $\xi\mapsto e_\xi$ is well defined all over $\mc H$. However, we regard the functions $e_\xi$ with $\xi\in\mc H\setminus\mc F$ as not belonging to $\mc A$, for they are not smooth in the sense that, given $E\in\mc E$,  $Ee_\xi$ might not be of the form $e_\eta$\lo, for any $\eta\in\mc H$.
\end{rk}
\begin{rk} \label{pathological characters}
The existence of pathological characters should be expected, for $\mc A$ can be infinite dimensional. In the next section we will define a natural topology on $\mc A$ which will enable us to exclude them, if needed.
\end{rk}

\begin{prop} \label{Xe_(xi) = e_(Xxi)}
One has that
\(
Ee_{\xi} = e_{E\xi}\loo, 
\)
for all  $\xi\in\mc F$ and $E\in\mc E$. In particular, $\mc A$ is a submodule of $\mc E^\dagger$.
\end{prop}
\begin{proof}
Indeed, given $F\in\mc E$,
\[
Ee_\xi(F) = \langle J\xi, \overline{E}{}^\dagger F\nu \rangle = \langle \overline{E}J\xi, F\nu \rangle = \langle JE\xi,F\nu \rangle = e_{E\xi}(F). \qedhere
\]
\end{proof}
By Proposition \ref{(Da)^dagger = bar D a^dagger}, Proposition \ref{Xe_(xi) = e_(Xxi)} and cyclicity, $\mc A$ is the  involutive subalgebra  of $\mc E^\dagger$ generated by the orbit of the state
\[
\omega\in\mc E^\dagger,\quad \omega(E) = e_\nu(E) =  \langle \nu, E\nu\rangle
\]
under the action of $\mc E$. 
Thus, the fundamental object is $\omega\in\mc E^\dagger$, and the representation $\pi$ can be recovered from it by GNS construction.

The affine scheme $M$ can be quite far from being a differential manifold, but it is at least \emph{reduced} (has no nilpotents) and \emph{irreducible} (has exactly one minimal prime ideal), as follows from the fact that $\mc E^\dagger\cong \C\lsem x_1\lo,\dots,x_n\rsem$ is an integral domain.

\begin{ej} 
Consider the case
\[
\mf g = \R X,\quad \mc F = \mc H = \C^2,\quad X = \left(\begin{smallmatrix} 0 & 1 \\ -1 & 0 \end{smallmatrix}\right),\quad \nu = \left(\begin{smallmatrix} 1 \\ 0 \end{smallmatrix}\right).
\]
Recall that the elements of $\mc A$ are, essentially, formal power series around~$\chi_0$\raisebox{-0.3ex},\, where $\chi_0(a) = a(1)$. Here, we have only the derivation $X$, which will be, say, with respect to the coordinate $x$. The power series corresponding to $\omega$ is computed as follows:
\begin{align*}
\omega|_{x=0} &\cong \chi_0(\omega) = \omega(1) = \langle\nu,\nu\rangle = 1, \\
\partial_x\omega|_{x=0} &\cong \chi_0(X\omega) = \langle \nu, X^\dagger\nu\rangle = \left\langle \left(\begin{smallmatrix} 1 \\ 0 \end{smallmatrix}\right), \left(\begin{smallmatrix} 0 & -1 \\ 1 & 0 \end{smallmatrix}\right) \left(\begin{smallmatrix} 1 \\ 0 \end{smallmatrix}\right)\right\rangle = 0, \\
\partial_x^2\omega|_{x=0} &\cong \chi_0(X^2\omega) = \langle \nu, (X^\dagger)^2\nu\rangle = \left\langle \left(\begin{smallmatrix} 1 \\ 0 \end{smallmatrix}\right), \left(\begin{smallmatrix} -1 & 0 \\ 0 &-1 \end{smallmatrix}\right) \left(\begin{smallmatrix} 1 \\ 0 \end{smallmatrix}\right)\right\rangle = -1, 
\end{align*}
and so on. Thus,
\(
\omega \cong \cos(x),
\)
as one might have anticipated, and $\mc A$ is the algebra of trigonometric polynomials. 
\end{ej}

\section{Analytic considerations}

As mentioned in Remark~\ref{pathological characters}, the existence of pathological characters should be expected in general. In this section we propose a topology for $\mc A$, and therefore a  notion of non-pathological character (continuity). Our choice is based on viewing $\mc A$ as (a subalgebra of) a quotient of the algebra of polynomial functions on $\mc H$. For more information regarding analytic functions on locally convex spaces, see~\cite{m:Maze84, m:Dine81}.

\subsection{Polynomials and power series on $\mc H$} 

Let $\mc H$ be a Hilbert space, together with a distinguished anti-unitary map $J$. 
Recall that the algebraic tensor product $\mc H^{\otimes n}$ comes equipped with the scalar product
\begin{equation} \label{scalar_product_H^(otimes n)}
\langle \xi_1\otimes \cdots\otimes \xi_n\raisebox{-0.2ex},\, \eta_1\otimes\cdots\otimes \eta_n\rangle = \langle \xi_1\raisebox{-0.2ex},\, \eta_1\rangle \cdots \langle \xi_n\raisebox{-0.2ex},\, \eta_n\rangle.
\end{equation}
The symmetric group $S_n$ acts on $\mc H^{\otimes n}$ by
\[
\sigma(\xi_1\otimes\cdots \otimes \xi_n) = \xi_{\sigma(1)}\otimes\cdots\otimes \xi_{\sigma(n)}\Lo.
\]
Let
\(
\mc H^n = S_n\backslash \mc H^{\otimes n}.
\)
An orbit $S_n(\xi_1\otimes\cdots\otimes\xi_n) \in\mc H^n$ will be written as $\xi_1\cdots\xi_n$\loo.
There is a canonical scalar product in $\mc H^n$, inherited from that of $\mc H^{\otimes n}$. Indeed, $\mc H^n$ can be identified with the subspace $P_+(\mc H^{\otimes n})$, where the projection
\[
P_+(\xi_1\otimes\cdots\otimes\xi_n) = \frac{1}{n!}\sum_{\sigma \in S_n} \sigma(\xi_1\otimes\cdots\otimes \xi_n).
\]
Then, we define
\(
\langle \xi_1\cdots\xi_n\lo, \eta_1\cdots \eta_n\rangle = \langle \xi_1\otimes \cdots\otimes\xi_n\lo, P_+(\eta_1\otimes\cdots\otimes\eta_n)\rangle.
\)
Now,
let
\[
\C[\mc H] = \bigoplus_{n\in\N} \mc H^n \subseteq \prod_{n\in\N} \mc H^n = \C\lsem\mc H\rsem.
\]
Given $(\varphi_n)\in \C[\mc H]$ and $(\psi_n) \in \C\lsem \mc H\rsem$, one has the (sesquilinear) pairing
\[
\langle (\varphi_n), (\psi_n) \rangle = \sum_{n\in\N} \langle \varphi_n\raisebox{-0.2ex},\, \psi_n \rangle,
\]
which is well-defined because the sum is finite. 

There is a multiplication  on~$\C\lsem\mc H\rsem$, defined by linear extension from
\[
\varphi\psi = P_+(\varphi\otimes\psi) \in \mc H^{n+m},\quad \varphi\in\mc H^n,\ \psi\in\mc H^m,
\]
where we are using the identification $\mc H^n\cong P_+(\mc H^{\otimes n})$. Explicitly, the product of $(\varphi_n), (\psi_n) \in \C\lsem\mc H\rsem$ is given by
\[
\bigl((\varphi_n)(\psi_n)\bigr)_n = \sum_{i+j=n} \varphi_i\psi_j\Lo.
\]
Thus, $\C\lsem\mc H\rsem$ is an algebra.
Note that $\C[\mc H]\subseteq \C\lsem\mc H\rsem$ is a subalgebra.
We think of $\C[\mc H]$ as a space of polynomial functions on $\mc H$, as follows: given $\Phi = (\varphi_n) \in \C[\mc H]$, its evaluation at $\xi\in\mc H$ is
\[
\Phi(\xi) = \left\langle \Phi, \Omega_{\xi} \right\rangle,\quad \Omega_\xi = (\xi^n)\in \C\lsem\mc H\rsem. 
\]
Correspondingly, we think of $\C\lsem\mc H\rsem$ as a space of formal infinite jets on $\mc H$.

\subsection{Norms on $\C[\mc H]$ and the weak nullstellensatz}

In the finite dimensional case, the weak nullstellensatz  implies  automatic continuity for characters, and there lies the proof's essential difficulty. However, take any discontinuous linear functional $\mc H\rightarrow \C$ and extend it to $\C[\mc H]$ as a character; it does not correspond to any point in~$\mc H$, and should not be continuous for any reasonable topology. Thus, we see that, in general, continuity is not automatic,\footnote{If, as it will, the topology of $\C[\mc H]$ is defined by a family of seminorms, then Michael's conjecture, hopefully proved in~\cite{m:Sten98}, would imply that the discontinuous characters above admit no extension to the completion of $\C[\mc H]$.} and we must make a choice of topology for~$\C[\mc H]$.

First, given $r>0$, consider the norm 
\[
\norm{(\varphi_n)}_r = \sum_{n\in\N} \norm{\varphi_n}r^n,\quad (\varphi_n)\in\C[\mc H].
\]
In order to identify the topological dual of $\C[\mc H]$ with respect to~$\norm\cdot_r$ recall that the completion of $\mc H^{\otimes n}$ with respect to its scalar product  is a Hilbert space, which we will denote by $\mc H^{\boxtimes n}$ (in general, $\boxtimes$ will be the Hilbert space tensor product).
Since the action of $S_n$ on $\mc H^{\otimes n}$ is by isometries, it extends by continuity to~$\mc H^{\boxtimes n}$. We define
\[
\mc H^{\boxdot n} = S_n\backslash \mc H^{\boxtimes n}.
\]
Then, it is easily seen that the topological dual of $\C[\mc H]$ with respect to $\norm{\cdot}_r$ can be identified with
\[
\Set{ (\varphi_n) \in \prod_{n\in\N} \mc H^{\boxdot n} | \norm{\varphi_n}\leq O(r^n) } 
\]
and, in particular,  contains 
\(
\Set{\Omega_{\xi} \in \C\lsem\mc H\rsem | \xi\in \mc H_r}.
\)

\begin{prop}
The completion $\mc O(\mc H_r)$ of $\C[\mc H]$ with respect to $\norm{\cdot}_r$ is a Fréchet algebra. 
\end{prop}
\begin{proof}
Indeed, since 
\[
\norm{\varphi\otimes\psi}_{\mc H^{\otimes n}\otimes \mc H^{\otimes m}} = \norm{\varphi}_{\mc H^{\otimes n}} \norm{\psi}_{\mc H^{\otimes m}}\LOO, \quad \varphi\in \mc H^{\otimes n},\ \psi\in\mc H^{\otimes n},
\]
the  product $\mc H^n\otimes \mc H^m\rightarrow \mc H^{n+m}$ satisfies
\[ 
\norm{\varphi\psi}_{\mc H^{n+m}} \leq \norm{\varphi}_{\mc H^n}\norm{\psi}_{\mc H^m}\LOO. 
\]
Thus, given $(\varphi_n), (\psi_n) \in \C[\mc H]$,
\[
\norm{(\varphi_n)(\psi_n)}_r \leq \sum_{i,j\in\N} \norm{\varphi_i}\norm{\psi_j} r^{i+j} = \norm{(\varphi_n)}_r \norm{(\psi_n)}_r\LO.  \qedhere
\]
\end{proof}

Now, let $r'>r$. Clearly,  $\norm\cdot_{r}$ is dominated by $\norm\cdot_{r'}$\LO.  Thus, we get a Banach algebra morphism $\mc O(\mc H_{r'})\rightarrow \mc O(\mc H_{r})$. 
\begin{defn}
$\mc O(\mc H)$ will denote the Michael-Arens algebra obtained as the projective limit $\mc O(\mc H) = \varprojlim \mc O(\mc H_r)$.
\end{defn}

\begin{prop}
There is a bijective correspondence between continuous characters of $\C[\mc H]$ and points of $\mc H$, given by 
\[
\xi\in\mc H \mapsto\Omega_{\xi} \in\mc \C[\mc H]^*.
\]
\end{prop}
\begin{proof}
By restriction, a character $\chi\in\C[\mc H]^*$ defines a continuous, linear function $\mc H\rightarrow \C$ which, by Riesz' theorem, corresponds to a vector, say $J^*\xi\in\mc H$. 
Then, clearly 
\(
\chi(\Phi) = \langle J^*\Omega_{\xi}\lo, \Phi\rangle = \langle J\Phi, \Omega_\xi\rangle
= \Phi(\xi). 
\)
\end{proof}

\subsection{The embedding $M\hookrightarrow\mc H$}

Since $\mc E^\dagger$ is commutative, the assignment $\xi\in \mc F\mapsto e_\xi\in \mc A$ extends to an involutive algebra epimorphism $\pi:\C[\mc F]\rightarrow \mc A$. 
The subspace $\C[\mc F]$ is topologized as a subset of $\C[\mc H]$, and we equip $\mc A\cong \C[\mc F]/\ker\pi$ with the quotient topology. 
Thus, if we define $M$ to be the set of \emph{continuous} characters of~$\mc A$, we get an embedding $M\hookrightarrow \text{Spec } \C[\mc F]$ given by $\chi\mapsto \chi\circ\pi$.
Now, a continuous character of $\C[\mc F]$ extends uniquely by continuity to all of $\C[\mc H]$ and, therefore, $\text{Spec } \C[\mc F] \cong \mc H$. In this way, we have obtained an embedding $M\hookrightarrow \mc H$. 
\begin{rk}
It would be satisfying to have a theory of affine schemes with Arens-Michael structure algebras. A proposal in that sense is made in~\cite{m:Varg12}.
\end{rk}

We close by stating an important, obvious question that arises: is $M$ a smooth manifold? If not, a study of its possible singularities should shed some light on the structure of non-integrable representations of real Lie algebras.

\bibliographystyle{amsplain}
\bibliography{math}

\end{document}